\documentclass[12pt]{article}

\usepackage[dvipsnames]{xcolor}
\usepackage[all]{xy}
\usepackage{amssymb}
\usepackage{amsmath}
\usepackage{amsthm}
\usepackage{amscd}
\usepackage{amsfonts}
\usepackage{mathrsfs}
\usepackage{mathtools}
\usepackage{mathabx}
\usepackage{enumerate}
\usepackage{times}
\usepackage[pdftex]{hyperref}
\usepackage[most]{tcolorbox}
\usepackage{titletoc}
\usepackage{moresize}
\usepackage{anyfontsize}
\usepackage[indentafter]{titlesec}
\usepackage{cleveref}
\usepackage[utf8]{inputenc}
\usepackage[T1]{fontenc}
\usepackage[margin=2.5cm]{geometry}
\usepackage{setspace}

\theoremstyle{plain}
\newtheorem{theo}{Theorem}[section]
\newtheorem{prop}[theo]{Proposition}
\newtheorem{lemma}[theo]{Lemma}
\newtheorem{coro}[theo]{Corollary}

\newtheorem{rem}[theo]{Remark}
\newtheorem{question}[theo]{Question}

\newtheorem*{thm1}{Theorem 1}

\theoremstyle{definition}
\newtheorem{df}[theo]{Definition}

\newtheoremstyle{TheoremNum}
{\topsep}{\topsep}              
{\itshape}                      
{}                              
{\bfseries}                     
{.}                             
{ }                             
{\thmname{#1}\thmnote{ \bfseries #3}}
\theoremstyle{TheoremNum}

\newcommand{\F}{\mathbb{F}}
\newcommand{\A}{\mathcal{A}}
\newcommand{\B}{\mathcal{B}}

\newcommand{\Z}{\mathbb{Z}}
\newcommand{\C}{\mathcal{C}}
\newcommand{\G}{\mathcal{G}}
\newcommand{\calH}{\mathcal{H}}

\newcommand{\calS}{\mathcal{S}}

\newcommand{\bfA}{\mathbf{A}}
\newcommand{\bfB}{\mathbf{B}}
\newcommand{\bfG}{\mathbf{G}}
\newcommand{\bfH}{\mathbf{H}}

\providecommand{\keywords}[1]
{\small	\textbf{Keywords:} #1.}

\makeatletter
\newsavebox{\@brx}
\newcommand{\llangle}[1][]{\savebox{\@brx}{\(\m@th{#1\langle}\)}%
  \mathopen{\copy\@brx\mkern2mu\kern-0.9\wd\@brx\usebox{\@brx}}}
\newcommand{\rrangle}[1][]{\savebox{\@brx}{\(\m@th{#1\rangle}\)}%
  \mathclose{\copy\@brx\mkern2mu\kern-0.9\wd\@brx\usebox{\@brx}}}
\makeatother

\everymath{\displaystyle}

\definecolor{mygreen}{rgb}{0, 0.3, 0}
\definecolor{myred}{rgb}{0.4,0,0}
\definecolor{myblue}{rgb}{0.1,0.1,0.6}

\hypersetup{
     colorlinks=true,
     linkcolor=mygreen,
     filecolor=myblue,
     citecolor = myred,      
     urlcolor=myblue,
}

\title{Prosoluble subgroups of free profinite products}
\author{ Pavel A. Zalesskii}
\date{\today}

\begin{document}

\maketitle

\begin{abstract} We describe finitely generated and second countable prosoluble subgroups of free profinite products. We  also give a description of relatively projective prosoluble groups.

\end{abstract}


\section{Introduction}

The  result that describes the algebraic structure of subgroups of free products  obtained by  Kurosh \cite{kurosh}  in 1934 is regarded as one of the fundamental results in combinatorial group theory.
Unfortunately,  the profinite analogue of the Kurosh Subgroup Theorem
does not hold for free products of profinite groups and so the algebraic structure of closed subgroups of free profinite products is not clear. Therefore   it is reasonable to study the  structure of the most important classes of subgroups of free profinite products. One of the most natural and important class of profinite groups is the class of prosoluble groups, as they play the same role in the profinite group theory  as soluble subgroups in the theory of finite groups.

The study of prosoluble subgroups of a free profinite product was initiated in middle 90s of the last century
 by  Jarden and Pop in \cite{J} and \cite{pop}, respectively;  the interest in the question has origin in Galois theory  (cf. \cite{pop}, \cite{HJ}, \cite{J}). In fact, the positive solution of  \cite[Problem 18]{Ja}, stating that a free profinite product of absolute Galois groups, is an absolute Galois group strengthens  this motivation (see \cite{K} for a description,  the statement and the history of the result).
 
 Pop  established the following necessary conditions for a prosoluble subgroup to be the closed subgroup of
a free product of profinite groups.

\begin{thm1}\label{pop}
Let $H$ be a prosoluble subgroup of a free profinite product $G=\coprod_{i\in I} G_i$.  Then
 one of the following statements holds:
\begin{enumerate}[(i)]
\item There is a prime $p$ such that $H \cap g G_i g^{-1}$ is pro-$p$ for every $i \in I$ and $g \in G$;

\item All non-trivial intersections $H \cap g G_i g^{-1}$ for $i \in I$ and $g \in G$ are finite and conjugate in $H$;

\item A conjugate of $H$ is a subgroup of $G_i$ for some $i \in I$.
\end{enumerate}
\end{thm1}

In the same paper Pop asks whether or not the situation described in (ii) of Theorem \ref{pop} does occur. Guralnick and Haran in \cite{GH} give an
affirmative answer to this question. Namely they prove that profinite Frobenius groups $\widehat\Z_\pi\rtimes C$, where $\pi$ is a set of primes and $C$ is cyclic finite, can be realized as a subgroup of a free profinite product. 

Our first objective  is to show that in fact, profinite Frobenius groups $\widehat\Z_\pi\rtimes C$ are the only possibility in Theorem  \ref{pop}(ii).

\begin{theo}\label{prosoluble general int}
Let $G=\coprod_{t\in T} G_t$  be a free profinite product of profinite groups over a profinite space $T$ and $H$  a prosoluble subgroup of $G$ that is not contained  in any conjugate $G_t^g$ of a free factor $G_t$.  If  the intersections $H \cap G_t^g, t\in T, g\in G$ are not all pro-$p$, then  $H\cong \widehat \Z_\pi\rtimes C$ is a  profinite Frobenius group (i.e $C$ is cyclic, $|C|$  divides $p-1$ for any $p\in \pi$ and $[c,z]\neq 1$ for all $1\neq c\in C, 0\neq z\in \widehat\Z_\pi$).
\end{theo}

This gives the complete description of prosoluble subgroups whose intersections with the conjugates of the free factors involve two different primes.

The case when all such intersection are pro-$p$ was studied by  Ersoy and  Herfort \cite{EH}. They showed   \cite[Theorem 1.4]{EH}  that a free profinite product  of pro-$p$ groups contains a free prosoluble product of conjugates of these pro-$p$ groups. 

Ersoy and  Herfort posed in the same paper the following questions. 

\begin{question}\cite[Question 3.2]{EH}\label{subgroup}. Let $G = A \amalg B$ be the free profinite product of pro-p groups $A$ and $B$. Is
every prosoluble subgroup $H$ of $G$ isomorphic to a closed subgroup of the free prosoluble product $A \amalg^s B$?\end{question}

\begin{question}\cite[Problem 4.5]{EH}\label{retract}. Let $A$ be a pro-p group and $P$ a projective prosoluble group. Does the free
profinite product $G := A \amalg P$ admit a prosoluble retract isomorphic to  a free prosoluble product $A \amalg_s P$?\end{question}

The second objective of the paper is to answer these questions. Namely, in Theorem \ref{retracts} we answer positively  Question \ref{retract}. Theorem \ref{subgroups} answers positively Question \ref{subgroup} if $H$ is   second countable, in particular, if $H$ is finitely generated.
This practically completes a description of prosoluble subgroups of a free profinite product. Indeed, the answers to the above questions show that for a subgroup $H$ of a free profinite product $\coprod_{x\in X} G_x$ either $H$ is contained in a conjugate of a free factor or it is a Frobenius group or it is a subgroup of a free prosoluble product of pro-$p$ groups $H\cap G_i^{g_i}$ with
 no restrictions.

 \begin{theo} \label{characterization fg} Let $G=\coprod_{x \in X} G_x$ be  a  free profinite product  over a profinite space $X$. A finitely generated prosoluble group $H$  is isomorphic to a subgroup of $G$  if and only if one of the following holds:
 
 \begin{enumerate} 
 
 \item[(i)]  $H$ is isomorphic to a subgroup of some  free prosoluble product $H_s=\coprod_{i=1}^n{}^s\  H_i$ of finitely generated pro-$p$ groups $H_i$  such that each $H_i$ is isomorphic to  a subgroup of some  $G_x$ and $\sum_{i=1}^n d(H_i)\leq d(H)$ ;
 \item[(ii)] $H$ is isomorphic to a profinite  Frobenius group  $\widehat \Z_\pi\rtimes C$, with $C$ a finite cyclic subgroup of $G_x$ for some $x\in X$;
\item[(iii)] $H$ isomorphic to a subgroup of  a free factor $G_x$.
\end{enumerate} 
 
 \end{theo} 

In fact Theorem \ref{characterization fg} valid also for a second countable subgroup $H$ (see Theorem \ref{characterization}), and  the hypothesis of second countability  is due to the open question posed by  Haran \cite[Problem 5.6 (a)]{H}. In this paper Haran introduced the notion of relatively projective  profinite groups and proved that a subgroup $H$ of a free profinite product $\coprod_{x \in X} G_x$ is projective relative to the family $\{H\cap G_x^g\mid x\in X, g\in G\}$.  Haran   \cite[Problem 5.6 (a)]{H} asks whether a profinite group which is projective relatively to the family of subgroups $\{G_t\mid t\in T\}$ embeds into a free profinite product of groups isomorphic to $G_t, t\in T$. This question was solved in \cite{HZ} for relatively projective pro-$p$ groups but for profinite and prosoluble relatively projective groups it is still open.

The next theorem is a characterization of  prosoluble relatively projective groups.

\begin{theo}\label{relatively projective} Let $G$ be a prosoluble group and $\{ G_t\mid t\in T\}$  a continuous family of its subgroups closed for conjugation. Then $G$ is  projective relative to $\{ G_t\mid t\in T\}$ if and only if one of the following holds:

\begin{enumerate}

\item[(i)] there is a prime $p$ such that  all $G_t, t\in T$ are pro-$p$ and $G$ is $\calS$-projective relative to $\{ G_t\mid t\in T\}$, where $\calS$ is the class of all finite soluble groups;

\item[(ii)]   $G= \widehat\Z_\pi\rtimes C$ is a  profinite Frobenius group (i.e. $C$ is cyclic, $|C|$  divides $p-1$ for any $p\in \pi$ and $[c,z]\neq 1$ for all $1\neq c\in C,  
0\neq z\in\widehat \Z_\pi$),  $T_0$ is closed in $T$,   $\{ G_t\mid t\in T_0\}=\{C^z\mid z\in \widehat\Z_\pi\}$ and $|T_0/G|=1$.

\item[(iii)] $|T_0|=1$ and $G_t=G$ for the unique $t\in T_0$.

\end{enumerate}  

\end{theo}

\begin{coro}\label{finitely generated} If $G$ is finitely generated  with minimal number of generators $d=d(G)$, then in (i) of Theorem \ref{relatively projective} in addition one has: $T_0=\{ t\in T\mid G_t\neq 1\}$ is closed in $T$ and $|T_0/G|\leq d$. Moreover, $\sum_t d(G_t)\leq d(G)$, where $t$ runs through representatives of $G$-orbits in $T_0$.

\end{coro}

From Theorem \ref{relatively projective} we also deduce the corollary that generalizes \cite[Theorems 1.4 and 4.4]{EH}.

\begin{coro}\label{herfort} Let $G=\coprod_{x\in X} G_x \amalg P$ and $G=\coprod_{x\in X}{}^s G_x \amalg^s P$ be  free profinite and free prosoluble products of pro-$p$ groups $G_x$  over a profinite space $X$ and a prosoluble projective group $P$. Then the natural epimorphism $\alpha:G\longrightarrow G_s$ that sends identically the factors of the free profinite product to the factors of
the free prosoluble product, admits a retraction $r:G_s\longrightarrow G$ such that $r(G_x)$ is conjugate to $G_x$ for every $x\in X$.

\end{coro}

Unless said otherwise, all groups in this paper are profinite,  subgroups are closed,  homomorphisms are continuous and  groups are generated in the topological sense.

\section{Case (ii) of Pop's theorem}

We begin this section with an easy generalization of \cite[Theorem 3.2]{HZ}.

\begin{lemma}\label{herfort-Ribes} Let $F$ be a  non-abelian free pro-$p$ group  and $G$ a finite soluble subgroup of $Aut(F)$ of order coprime to $p$. If  the subgroup of fixed elements $Fix_G(F)\neq F$, then $Fix_G(F)$ is infinitely generated.\end{lemma}

\begin{proof}  We use induction on $|G|$. 

Suppose $G$ is of prime order and suppose on the contrary $Fix_G(F)$ is finitely generated. Let $f\in F\setminus Fix_G(F)$ be an element of $F$ and $H=\langle Fix_G(F), g(f)\mid g\in G\rangle$. Then $H$ is a finitely generated $G$-invariant subgroup of $F$ and so  by  \cite[Theorem 3.2]{HR} $Fix_G(H)$ is infinitely generated. As $Fix_G(F)=Fix_G(H)$ we arrive at contradiction.
This gives us the base of induction.

Suppose now that  $|G|$ is not of prime order and  let $A$ be a proper normal subgroup of $G$. Then either $Fix_A(F ) = F$ or, by the induction hypothesis $Fix_A(F)$ is infinitely generated.   In particular, $Fix_A(F)$ is non-abelian. Note that  if $Fix_G(F) = Fix_A(F )$,
 then $Fix_A(F )\neq F$ , so $Fix_G(F) = Fix_A(F )$ is infinitely generated. Otherwise,  by the induction hypothesis again $Fix_{G/A}(Fix_A(F))=Fix_G(F)$ is infinitely generated as required. 

\end{proof} 

\begin{lemma}\label{projective} Let $G$ be a projective group and $S$ a Sylow $p$-subgroup of $G$. If $rank(S)=r$, then there exists an open subgroup $U$ of $G$ containing $S$ such that for any open subgroup $V$ of $U$ the maximal pro-$p$ quotient $V_p$ of $V$ is free pro-$p$ of rank $\geq r$. Moreover, if $r=1$ then $V=R_p(V)\rtimes \Z_p$ with $R_p(V)$ coprime to $p$.

\end{lemma}

\begin{proof}  Since $S$ is the intersection of all open subgroups $U$ containing $S$, we can view this as the inverse limit $S=\varprojlim_{S\leq  U\leq _o G} U$. It follows that $S=S_p=\varprojlim_{S\leq  U\leq _o G} U_p$ is the inverse limit of the maximal pro-$p$ quotients $U_p$ of $U$ (cf. \cite[Satz 3]{N} for example). Note that $U_p$ are free pro-$p$ (cf. \cite[Proposition 7.7.7]{RZ}) and $S$ is also a $p$-Sylow subgroup of $U$. Since an epimorphism maps $p$-Sylow subgroups  onto $p$-Sylow subgroups we deduce that $rank(U_p)\leq rank(S)$. Since $rank(S)=r$,   $U_p$ must be  free pro-$p$ of rank $r$  for some $U$. 

\medskip

Let $V$ be an open subgroup of $U$. Then the image of $V$ in $U_p$ is free pro-$p$ of rank $\geq rank(U_p)=r$ (see \cite[Theorem 3.6.2 (b)]{RZ}). But this image is a quotient of $V_p$. So $V_p$ is  free pro-$p$ of rank $\geq r$.  

\medskip

For the last part of the statement, assume $S$ is infinite cyclic. Then $V_p\cong \Z_p$ (by the first paragraph) and   the natural epimorphism $V\longrightarrow V_p$ to its maximal pro-$p$ quotient splits; so its kernel $R_p(V)$ intersects $S\cap V$ trivially. Hence $R_p(V)$ is coprime to $p$.
\end{proof}




The proof of the following lemma was obtained in communication with Pavel Shumyatsky.

\begin{lemma}\label{frattini argument1} Let $G=K\rtimes C$ be a semidirect product of a profinite group $K$ and a finite group $C$ such that $C_{K}(C)=1$. Let $N\leq K$ be  normal in $G$ and $L$ be the preimage of $C_{K/N}(CN/N)$ in $K$. Suppose that every $L$-conjugate of $C$   is $N$-conjugate to $C$. Then  $C_{K/N}(CN/N)=1$.
\end{lemma}

\begin{proof}  By hypothesis  for any $l\in L$ there exists an element of $n\in N$ such that $C^l=C^n$. This means that $N_L(C)N=L$.  As $N_L(C)\leq N_K(C)=C_K(C)=1$  we deduce that $N=L$. This means precisely that $C_{K/N}(CN/N) = 1$.

\end{proof}

\begin{lemma}\label{soluble} Let $p\neq l$ be primes and let $C=A\rtimes C_l$ be a semidirect product of elementary abelian $p$-group $A$  and a cyclic group of order $l$. Let $G=K\rtimes C$ be a semidirect product of a projective profinite group $K$ and $C$ such that $C_K(c)=1$ for every $1\neq c\in C$. Suppose that  for any $C$-invariant subgroup $L$ of $K$   every maximal finite subgroup  of $L\rtimes C$  is $L$-conjugate to $C$.  Then $G$ is soluble.
\end{lemma}

\begin{proof}   We shall first show that a $q$-Sylow subgroup of $K$ is cyclic for any $q\neq p$. Suppose not. Then by Lemma \ref{projective} there exists an open subgroup $U$ of $K$ such that the  maximal pro-$q$ quotient  of the core of $U$ in $G$ is non-cyclic free pro-$q$ and so replacing $U$ by its core  we may assume that $U$ is normal.  Let $R_q(U)$ be the kernel of the natural epimorphism of $U$ onto $U_q$. Since $R_q(U)$ is characteristic, it is normal in $G$ and  $UC/R_q(U)=U_q\rtimes C$. As $U_q$ is non-abelian free pro-$q$ and $C_{U_q}(A)\neq U_q$ by hypothesis,  we deduce from Lemma \ref{herfort-Ribes} that $C_{U_q}(A)$ is infinitely generated; in particular, $C_{U_q}(A)$ is not trivial. 

Let $L$ be the preimage of $C_{U_q}(A)$ in $G$. By hypothesis, all maximal finite subgroups of $L\rtimes C$ are conjugate to $C$   and  all maximal finite subgroups of $R_q(U)\rtimes C$ are conjugate to $C$. Then all maximal finite $p$-subgroups of $L\rtimes A$ are conjugate to $A$ and  all maximal finite $p$-subgroups of $R_q(U)\rtimes A$ are conjugate to $A$. As $R_q(U)A$ is normal in $LA$, it follows that any $L$-conjugate of $A$ is  in $R_q(U)A$ and so is $R_q(U)$-conjugate to $A$. Thus   $L\rtimes A$ and $R_q(U)\rtimes A$ satisfy the premises of Lemma \ref{frattini argument1} and applying it one deduces that  $C_{U_q}(A)$ is trivial, a  contradiction. Thus $q$-Sylow subgroups of $K$ are cyclic. 

Now by Lemma \ref{projective} there exists an open normal subgroup $V$ of $G$ contained  in $K$ such that  $V\cong R_q(V)\rtimes \Z_q$ and $R_q(V)$ is coprime to $q$, in particular  putting $q=l$ we have that $R_l(V)$ is coprime to $l$. Hence $R_l(V)\rtimes C_l$ is a profinite Frobenius group (\cite[Theorem 4.6.9 (d)]{RZ}) and so  $R_l(V)$ is nilpotent (see \cite[ Corollary 4.6.10]{RZ}). Since $R_l(V)$ is projective, it is cyclic, in particular a $p$-Sylow subgroup of $R_l(V)$  is cyclic. As a torsion free virtually cyclic pro-$p$ group is cyclic, we deduce that  all Sylow subgroups of $K$ are cyclic.   Then $K$ is metacyclic by \cite[Exercise 7.7.8 (a)]{RZ} and therefore $G$ is soluble.    

\end{proof}




 We finish this section proving Theorem \ref{prosoluble general int}  for finitely many factors only and postpone the proof for   the  general free profinite product until the next section.

\begin{theo}\label{prosoluble}
Let $G=\coprod_{i=1}^n G_i$ be a free profinite product of profinite groups and $H$  a prosoluble subgroup of $G$ that is not a subgroup of a free factor up to conjugation.  If  the intersections $H \cap G_i^g$, $g \in G$, $i=1, \ldots, n$ are not all pro-$p$ for some prime $p$, then  $H\cong \widehat \Z_\pi\rtimes C$ is a  profinite Frobenius group (i.e $C$ is cyclic, $|C|$ divides $p-1$ for every $p\in \pi$ and $[c,z]\neq 1$ for all $1\neq c\in C, 0\neq z\in \widehat\Z_\pi$).
\end{theo}

\begin{proof} By Theorem \ref{pop} if $H\cap G_i^g$ for $g\in G$ and $ i=1, \ldots, n$ are not all pro-$p$ and none of them is $H$  then they are finite and  conjugate in $H$. Since for $i\neq j$ no non-trivial subgroup of $G_i$ can be conjugate in $G$ to a subgroup of $G_j$ (as can be easily seen looking at the quotient $\prod_{i=1}^n G_i$ of $G$), it follows that $H\cap G_i^g$ is non-trivial  for one $i$ only, say for $i=1$. Then the intersection $K=H\cap \llangle G_2, G_3, \ldots, G_n\rrangle $ of the normal closure of $G_2, \ldots, G_n$ in $G$ with $H$ is normal in $H$  and torsion free (as every finite subgroups is conjugate to a subgroup of a free factor \cite[Corollary 7.1.3]{R}). Thus  $H=K\rtimes C$ with $C= H\cap G_{1}^g$ for some $g\in G$. 

If $H$ is soluble then by  \cite[Corollary 7.1.8]{R} $H\cong \widehat\Z_\pi\rtimes C_n$ is a Frobenius profinite group, where $C_n$ is cyclic of order $n$ dividing $p-1$ for every $p\in\pi$. 

Thus to finish the proof we just need to prove that $H$ is soluble. Choose a minimal normal elementary abelian $p$-subgroup $A$ in $C$ and a cyclic subgroup $C_l$  of prime order $l$ coprime to $p$. 
Replacing  $C$ by $A\rtimes C_l$, if necessary, we may assume w.l.o.g.  that $C=A\rtimes C_l$.   By \cite[Corollary 7.1.5 (a)]{R}  $C_K(c)=1$ for every $1\neq c\in C$. By Theorem \ref{pop} all maximal finite subgroups of $H$ and of any subgroup containing $C$ are conjugate to $C$. Then by Lemma  \ref{soluble} $H$ is soluble.  \end{proof}

\section{Subgroups of free pro-$\C$ products}

In this section $\C$ will denote a class of finite groups closed for subgroups, quotients and extensions. We shall use it however mainly for the class of all finite groups or the class $\calS$ of all finite soluble groups.  

\begin{df}\label{continuous system} A system $E$ of subgroups of a pro-$\C$ group $G$ is a
map $\Sigma: T\longrightarrow Subgrps(G)$, where $T$ is a profinite
space and $Subgrps(G)$ denotes the set of all closed subgroups of
$G$. We shall find it convenient to write $G_t$ for $\Sigma(t)$. We shall say
that a system $E$ is continuous if, for any open subgroup $U$ of
$G$, the set $T(U) = \{t\in T\mid G_t\leq U\}$ is open in
$T$. A continuous system  $\{H_s\mid s\in S\}$ over a profinite space $S$  such that for every $s\in S$ there exists a unique $t\in T$ with $H_s\leq G_t$  will be called a continuous subsystem of the continuous system $\{G_t\mid t\in T\}$.\end{df}

 \begin{lemma}\label{systeproperties} (\cite[Lemma 5.2.1]{R})   Let  $\{G_t\mid t\in T\}$ be a system of subgroups of a
 pro-$\C$
group $G$. The following statements are equivalent:
\begin{enumerate}
\item[(i)] $\{G_t\mid t\in T\}$ is continuous. \item[(ii)]
$\{(g,t)\in G\times  T\mid g\in G_t\}$ is closed in $G \times T$.
\item[(iii)] $E = \bigcup_{t\in T} G_t$ is closed in G.
\end{enumerate}
\end{lemma}

\begin{df}\label{def:freeprod}  A pro-$\C$ group $G$ is said to be a free
(internal) pro-$\C$ product of a continuous system $\{G_t\mid t\in
T\}$ of subgroups  if

(i) $G_t \cap G_s = 1$ for all $t\neq s$ in $T$;

(ii) for every pro-$\C$ group $H$, every continuous map $\beta:
\bigcup_{t\in T} G_t\longrightarrow H$, whose restriction
$\beta_t:=\beta_{|G_t}$ is a homomorphism for every $t\in T$,
extends uniquely to a homomorphism $\psi: G \longrightarrow
H$.\end{df}
We shall use  $\coprod_{t\in T} G_t$ to denote the free profinite product, $\coprod_{t\in T}{}^\C G_t$ to denote the free pro-$\C$ product and $\coprod_{t\in T}{}^{s} G_t$ to denote the free prosoluble product.  The reader can find in  \cite{R}  the proof of the existence and uniqueness of a free pro-$\C$ product as well as basic properties of it.

\bigskip

 We are ready to prove Theorem \ref{prosoluble general int}.

\bigskip

Proof of Theorem \ref{prosoluble general int}. Suppose $H$ is not Frobenius and for no prime $p$ the intersections $H \cap G_t^g, t\in T, g\in G$ are  all pro-$p$. By \cite[(1.7)]{M} (or by \cite[Lemma 4.5.4]{HJ}) for any clopen partition $T=V_1\cup V_2\cup \cdots \cup V_n$ we have $G=G(V_1)\amalg G(V_2) \amalg \cdots \amalg G(V_n)$, where $G(V_i)=\langle G_t\mid t\in V_i\rangle$, and denoting by $\Lambda$ all such clopen partitions of $T$ we have $G_t=\bigcap_{\lambda\in\Lambda} G(V_\lambda)$, where $V_\lambda$ is an element of the partition $\lambda$ containing $t$. By Theorem \ref{prosoluble} $H$ is contained in a conjugate $G(V_\lambda)^{g_\lambda}$ for each $\lambda\in \Lambda$ and this conjugate is unique since $G(V_i)\cap G(V_j)^k=1$ whenever $i\neq j$ or $k\not\in G_j$ by \cite[Theorem 9.1.12]{RZ}.    Then $H$ must be contained in the intersection $\bigcap_{\lambda\in \Lambda}G(V_\lambda)^{g_\lambda}$ of such conjugates  and so $H\leq G_t^{g}$ for some $t\in T, g\in G$.

\begin{rem} One can use also \cite[Theorem 5.3.4]{R} instead of \cite{M} and a projective limit argument from \cite[Remark 4.5.9]{HJ} to perform the proof of Theorem \ref{prosoluble general int}. 

\end{rem}

In the rest of the  section we shall use only finite sets of indices of a free pro-$\C$ product, thereby avoiding general notions, as discussed in Definition \ref{continuous system} and 
  Lemma \ref{systeproperties}. 
  
  \begin{lemma}\label{getting free free factor abstract}  Let $G= G_1 * G _2$ be a free (abstract)  product of non-trivial finite groups such that $|G_1|+|G_2|>4$.  Then for some $g\in G$ there exists a finite index subgroup  $U$ of $G$ that decomposes as  a
free  product $G_1 * G^g_2 * M *  F$  with  $F$  a non-trivial free  group. 

\end{lemma}

\begin{proof} Choose  $1\neq g_i\in G_i$ and put $g=g_1g_2$.  Then the  subgroup $H$ generated by $G_1$ and $G_2^g$ is  a free product $H=G_1* G_2^g$ and a proper subgroup of $G$ as no element from $G_2\setminus\{1\}$ is a product of elements of $G_1$ and $G_2^g$. Denote by $C_G$ the Cartesian subgroup  of $G$:  the kernel of the epimorphism onto the direct product of factors. Note that $C_G$ is free of finite index in $G$ and since  $G$ is not dihedral (by hypothesis) $C_G$ is free non-abelian.  As $C_H$ is free of the same rank as $C_G$,  $C_H$ is of infinite index in $C_G$ by the Schreier formula and therefore $[G:H]=\infty$.   Hence by the analog of M. Hall theorem for free products (see \cite[Theorem 1.1]{B}) $H$ is a free factor of a subgroup of finite index $U$ of $G$, i.e. $U=H* L$ and using the Kurosh subgroup theorem for $L$ we have a decomposition $L=*_{i=1}^2 *_{g_i} (L\cap  G_i^{g_{i}}) * F_0$, where  $F_0$ is a free  group. Here $F_0$ can be trivial. If this is the case, apply the above argument  to the decomposition $U=H*L$ to get a subgroup of finite index $U_1=H*L^u* L_1$ with $L_1=*_{i=1}^2 *_{g_i} (L_1\cap  G_i^{g_{i}}) * F_1$, where $F_1$ is free (possibly trivial).  Putting $V_1=L^u*L_1$ we get $U_1=H*V_1$ with $V_1$ having a nontrivial free normal subgroup $F_2$ of finite index.   Setting $U_2$ to be the kernel of the epimorphism $U_1\longrightarrow  V_1/F_2$ that sends $H$  to the trivial group and $V_1$ naturally onto $V_1/F_2$ we get the Kurosh decomposition  $U_2=H*F_2* M$ with $F_2$ non-trivial. 
Thus for certain subgroup of finite index $U$ we achieve a nontrivial decomposition $U=H*M* F$, where $F$ is non-trivial. This finishes the proof.

\end{proof}

\begin{lemma}\label{getting free free factor}  Let $G= G_1 \amalg G _2$ be a free profinite product of non-trivial profinite groups such that $|G_1|+|G_2|>4$.  Then for some $g\in G$ there exists an open subgroup  $U$ of $G$ that decomposes as  a
free profinite product $G_1 \amalg G^g_2\amalg U\amalg F$, where $F$ is a non-trivial free profinite group.  

\end{lemma}

\begin{proof} Suppose first that $G_1,G_2$ are finite and so $G$ is the profinite completion of the abstract free product $G^{abs}=G_1 *G_2$. By Lemma \ref{getting free free factor abstract} there exists $g\in G^{abs}$ and a finite index subgroup $U^{abs}$ of $G^{abs}$ such that $G_1 * G^g_2 * M^{abs} * F^{abs}$  with  $F^{abs}$  a non-trivial free  group (of finite rank). It follows that $U=\widehat U^{abs}=G_1\amalg G_2^g\amalg M \amalg \widehat F^{abs}$ is a free profinite product with $\widehat F^{abs}$  non-trivial  free profinite, as required.

\medskip
Suppose now  $G_1$, $G_2$ are arbitrary. Let $Q_i$ be a non-trivial finite quotient of $G_i$ and $$f:G\longrightarrow Q=Q_1 \amalg Q_2$$ the  epimorphism induced by the natural epimorphisms $f_i:G_i\longrightarrow Q_i$.   By the previous paragraph there exists an open subgroup $U_Q$ such that $U_Q=Q_1\amalg Q_2^q\amalg M\amalg F_Q$ for some  $q\in Q$ and a non-trivial free profinite group $F_Q$. Put $U=f^{-1}(U_Q)$. Then by a profinite version of the Kurosh subgroup theorem (cf. \cite[Theorem 9.1.9]{RZ}) $U=(G_1\amalg G_2^g)\amalg \left(\coprod_{i=1}^2{}\coprod (U\cap  G_i^{g_{i}})\right) \amalg F$, where   $F$ is a free profinite group. As $ker(f)$ is generated by $ker(f_1),\ker(f_2)$ as a normal subgroup,  the kernel of the natural projection  $U\longrightarrow F$ contains $ker(f)$ and so one has the following commutative diagram of epimorphisms:
$$\xymatrix{U\ar[r]\ar[d]^f&F\ar[d]\\
            U_Q\ar[r]& F_Q}$$
Since $F_Q$ is non-trivial so is $F$. The proof is finished.
 
\medskip

\end{proof}

To prove a prosoluble version of Lemma \ref{getting free free factor}
we shall need  the following simple lemma about finite soluble groups.

\begin{lemma}\label{fater} Let $p, q$ be prime numbers. Let  $H=C_p\times C_q$ be a cyclic group of order $pq$ if $p\neq q$ and $H=C_p$ if $p=q$. Then there exists a prime $r$ such that $pq| r-1$ and a semidirect product
$R= \F_{r}\rtimes H$ of a  field 
$\F_{r}$ of $r$ elements  with  $H$  such that 
$R$ is generated by an element of order $p$ and an element of order $q$.\end{lemma}

\begin{proof} Since an
arithmetic progression $\{1+pqk | k=1, 2 \ldots \}$ contains infinitely many primes by
Dirichlet’s theorem, there exists a prime $r$ such that $pq|r-1$.
Consider a semidirect product
$R= \F_{r}\rtimes H$ with the faithful action of $H$ by multiplication by units of $\F_r$. Then for any non-zero element $m \in F_{r}$ we
 get $R=\langle C_p , C_q^m\rangle$. Indeed, the generators $c_1$ of $C_p$ and $c_2$ of $C_q^m$ have order $p$ and $q$ respectively and their commutator $[c_1,c_2]$ has order $r$. Therefore $\langle C_p , C_q^m\rangle$  contains elements of order $p, q$ and $r$ and so must coincide with $G$. \end{proof}

\begin{lemma}\label{embeds} Let $G= G_1 \amalg^s G _2$ be a free prosoluble product of non-trivial prosoluble groups such that $|G_1|+|G_2|>4$.  Then for some $g\in G$ there exists an open subgroup  $U$ of $G$ that decomposes as  a
free prosoluble product $G_1 \amalg^s G^g_2\amalg^s U_1\amalg^s U_2 \amalg^s F_s$, where $F_s$ is a non-abelian free prosoluble group and $U_i=\coprod_{ g_i}{}^s (U\cap  G_i^{g_{i}})$. Moreover, if $G_i$ is infinite then $U_i$ is non-trivial.\end{lemma}

\begin{proof} Let $C_p$ and $C_q$ be  quotient groups of $G_1$ and $G_2$ of prime orders $p$ and $q$ respectively.

 Let  $R=\F_{r}\rtimes H$ be a finite group from Lemma \ref{fater}  generated by elements $a,b$ of order $p$ and $q$ and assume w.l.o.g that $a\in H$.  Let $f_1:G_1 \longrightarrow \langle a\rangle$ and $f_2:G_2\longrightarrow \langle b \rangle$ be the natural epimorphisms  with kernels $N_1$ and $N_2$.  Consider an epimorphism $f :  G \longrightarrow R$ extending $f_1$ and $f_2$. Then $ U= f^{-1}(H)$ is a
proper open subgroup of $G$ and so by a prosoluble version of the Kurosh subgroup theorem (cf. \cite[ Theorem 9.1.9]{RZ}) $U=\coprod_{i=1}^2{}^s\coprod_{ D_i}{}^s (U\cap  G_i^{g_{i}}) \amalg F_s$, where $D_i=U\backslash G/G_i$, $g_i$ runs through representatives of the double cosets  $D_i$ and $F_s$ is a free prosoluble group of rank $1+[G:U]-\sum_{i=1}^2|U\backslash G/G_i|$. 

Observe that $G_1, G_2^g\leq U$ for some $g\in f^{-1}(c)$, where $c$ is an element conjugating $b$ into $H$.
Therefore we can rewrite $U=G_1  \amalg^s G_2^{g} \amalg^s  U_1\amalg^s U_2 \amalg^s F_s$. 
Note that $|U\backslash G/G_1|=|H\backslash R/\langle a\rangle|$ and taking into account that $Hra=Har^a$ we can see that the action of $\langle a\rangle$ on $H\backslash R$ is isomorphic to the action of $\langle a\rangle$ on $\F_r$, so $|H\backslash R/\langle a\rangle|=\frac{r-1}{p}+1\geq 2$. Similarly,  $|U\backslash G/G_2|=\frac{r-1}{q}+1\geq 2$ and so if $G_i$ is infinite, the decomposition of $U_i$ into a free profinite product contains at least one non-trivial  factor $U\cap  G_i^{g_{i}}=N_i^{g_i}$ distinct from $G_1$, $G_2^g$.

 We show that $F_s$ is non-trivial if $|H|>2$. Indeed, $$1+[G:U]-\sum_{i=1}^2|U\backslash G/G_i|=1+r-\frac{r-1}{p}-\frac{r-1}{q}-2)=(r-1)(1-\frac{1}{p}-\frac{1}{q})$$ which is not $0$ unless $H$ is of order $2$.
 
  If $H$ in Lemma \ref{fater} has order $2$ then $p=q=2$. In this case  $U_1\amalg^s U_2$ is non-trivial, since at least one $G_i$ has order $>2$ and as $|U\backslash G/G_i|\geq 2$ we have a factor $U\cap G_i^{g_i}\neq 1$ in $U_i$. 

 To get $F_s$ non-trivial in this case we put $G'_1=G_1\amalg^s G_2^g$, $G'_2=U_1\amalg^s U_2$ and  write $U$ as $U=G'_1\amalg^s G'_2$.  Then applying the argument above to this decomposition we get an open subgroup $V\leq U$ such that $V=G'_1\amalg^s (G'_2)^{g'}\amalg^s V_1=G_1\amalg^s G_2^g\amalg^s (G'_2)^{g'}\amalg^s V_1$. If $(G'_2)^{g'}\amalg^s V_1$ has a quotient of order $p\neq 2$ then the above argument applied to $V$ produces a needed open subgroup with $F_s$ non-trivial. Otherwise, the kernel $K$ of the epimorphism $V\longrightarrow C_2$ to a group of order 2 that sends $G_1\amalg^s G_2^g$ to the trivial group and $(G'_2)^{g'}, V_1$ epimorphically to $C_2$ has the Kurosh decomposition $K=G_1\amalg^s G_2^g\amalg^s (K\cap (G'_2)^{g'})\amalg^s (K\cap V_1) \amalg^s F_s$ with $F_s\cong \Z_2$.

  Finally, to get a non-abelian factor which is free  we can put $G'_1=G_1\amalg^s G_2^g$ and apply the argument of the preceding paragraph  to  $G'_1 \amalg^s F_s$  getting an open subgroup $W=G'_1\amalg^s (F_s)^{w}\amalg^s F'_s=G_1\amalg^s G_2^{g}\amalg^s F_s^{w}\amalg^s F'_s$ for some $w\in W$ with non-trivial free prosoluble group $F'_s$, so that one has a non-abelian free factor $F_s^w\amalg^s F'_s$.   
  This finishes the proof. 
  
\end{proof}

\begin{lemma}\label{embedding} Let  $G=G_1\amalg^s G_2$ be a free prosoluble product of non-trivial prosoluble groups such that $|G_1|+|G_2|>4$. Then for some $g\in G$ the group $G$ contains $G_1\amalg^s G_2^g \amalg^s U$ such that $|U|=|G|$. \end{lemma}

\begin{proof} By Lemma \ref{embeds} $G$ contains a free prosoluble product  $$G_1 \amalg^s G^g_2\amalg^s U_1\amalg^s U_2\amalg^s F_s,$$  where $U_i$ is isomorphic to a free product of open subgroups of $G_i$, $i=1,2$ and  a non-abelian free prosoluble group $F_s$.
Moreover, $U_1$ and $U_2$ are non-trivial if $G_1,G_2$ are infinite. 
Hence if one of $G_i$ is infinite, say $G_1$ and   w.l.o.g.  $|G_1|\geq |G_2|$, then $|U_1|=|G|=|U_1 \amalg  U_2\amalg F_s|$ and so  the result is proved.

 If $G_1,G_2$ are finite, the result is obvious since $F_s\neq 1$.

\end{proof}

\section{Relatively projective groups}

In this section we shall follow the terminology \cite{HZ} that concerns relatively projective groups. In particular, the actions of  profinite groups on profinite spaces will be on the right. We also continue to use $\calS$ for the class of all finite soluble groups and $\C$  for an arbitrary class of finite groups closed for subgroups, quotients and extensions.

\begin{df}\label{def proj}
Consider the category of \textbf{profinite pairs}
$(G,\G)$,
where $G$ is a profinite group
and $\G$ is
a continuous family of closed subgroups of $G$,
closed under the conjugation in $G$.

A \textbf{morphism}
$\varphi \colon (G,\G) \to (A,\A)$
in this category
is a homomorphism $\varphi \colon G \to A$ of profinite groups
such that
$\varphi(\G) \subseteq \A$,
that is,
for every $\Gamma \in \G$ there is $\Delta \in \A$
such that $\varphi(\Gamma) \le \Delta$;
it is an \textbf{epimorphism}, if
$\varphi(\G) = \A$.

An \textbf{embedding problem} for $(G,\G)$
(cf. \cite[Definition 5.1.1]{HJ} or \cite[Definition 4.1]{H})
is a pair of morphisms
\begin{equation}\label{EP pairs}
\xymatrix{
& (G,\G) \ar[d]^{\varphi} \\
(B,\B) \ar[r]^{\alpha} & (A,\A) \\
}
\end{equation}
such that $\alpha$ is an epimorphism
and for every $\Gamma \in \G$
there exists $\Delta \in \B$
and a homomorphism $\gamma_\Gamma\colon \Gamma\to \Delta$
such that
$\alpha\circ\gamma_\Gamma=\varphi|_{\Gamma}$.

It is a {\it $\C$-embedding problem} if $G,B,A$ are pro-$\C$.
We say that \eqref{EP pairs} is \textbf{finite}, if $B$ is finite.
We say that \eqref{EP pairs} is \textbf{rigid},
if $\alpha$ is \textbf{rigid},
i.e.,
$\alpha|_{\Delta}$ is injective for every $\Delta \in \B$.

A \textbf{solution} of \eqref{EP pairs} is a morphism
$\gamma\colon (G,\G)\to (B,\B)$
such that $\alpha\circ\gamma=\varphi$.

We say that $G$ is $\C$-\textbf{projective relative to $\G$}, if
every finite $\C$-embedding problem \eqref{EP pairs} for $(G,\G)$
has a solution.
Equivalently (\cite[Corollary 5.1.5]{HJ}),
every finite rigid embedding problem \eqref{EP pairs} for $(G,\G)$ has a solution. If $\C$ is a class of all finite groups we simply write projective.
\end{df}

\begin{rem}\label{second countable embedding} By \cite[Lemma 3.5.1]{HJ} if $G$ is $\C$-projective then every $\C$-embedding problem \eqref{EP pairs} with $B$  second countable is solvable.

\end{rem}

\begin{df}
A {\it group pile} (or just {\it pile}) is a pair $\bfG = (G,T)$ consisting of a profinite group $G$, a profinite space $T$ and a continuous action of $G$ on $T$. The stabilizer of $t\in T$ will be denoted by $G_t$.
\end{df}

\begin{df}
A {\it morphism} of group piles $\alpha: (G,T) \longrightarrow (G',T')$ consists of a group homomorphism $\alpha: G \longrightarrow G'$ and a continuous map $\alpha: T \longrightarrow T'$ such that $\alpha(tg) = \alpha(t){\alpha(g)}$ for all $t \in T$ and $g \in G$. The above morphism is an {\it epimorphism} if $\alpha(G) = G'$, $\alpha(T) = T'$ and for every $t' \in T'$ there is $t \in T$ such that $\alpha(t) = t'$ and $\alpha(G_t) = {G'}_{t'}$. It is {\it rigid} if $\alpha$ maps $G_t$ isomorphically onto ${G'}_{\alpha(t)}$, for every $t \in T$, and the induced map of the orbit spaces $T/G  \longrightarrow T'/G'$ is a homeomorphism.
\end{df}

\begin{df}
An {\it embedding problem} for a pile $\bfG$ is a pair $(\varphi,\alpha)$  of morphisms of group piles, where $\varphi: \bfG \to \bfA$ and $\alpha: \bfB \to \bfA$, such that $\alpha$ is an epimorphism.

\begin{equation}\label{EPfree products}
\xymatrix{&&\bfG\ar[d]^{\varphi}\ar@{-->}[dll]_{\gamma}\\
             \bfB\ar[rr]_{\alpha}&& \bfA}
\end{equation} 

 It is a {\it $\C$-embedding problem} if $G,B,A$ are pro-$\C$. It is {\it rigid}, if $\alpha$ is rigid. A {\it solution} to the embedding problem is a morphism $\gamma: \bfG \to \bfB$ such that $\alpha \circ \gamma = \varphi$.
\end{df}

\begin{df}\label{def: projective}
A pile $\bfG=(G,T)$ is {\it $\C$-projective} if every finite rigid $\C$-embedding problem for $\bfG$ has a solution. In this case we say that $G$ is $\C$-projective relative to the family of point stabilizers $\{G_t\mid t\in T\}$.
\end{df}

\begin{rem}\label{matching definitions} Definition \ref{def: projective}  is  sligthly more restrictive than   Definition \ref{def proj};  it becomes equivalent if one adds the condition $G_t\neq G_{t'}$ whenever $t\neq t'\in T_0=\{t\mid G_t\neq 1\}$ (see \cite[Remark 5.8]{HZ}. In Definition \ref{def: projective} we assume that $G$ acts on a profinite space and the family is the family of point stabilizers. However, the notion does not depend on the choice of the space (as long as the family of point stabilizers is the same). Indeed,  if $(G,T)$ is $\C$-projective then $G_{t_1}\cap G_{t_2}=1$ for $t_1\neq t_2\in T$ and any pile $(G,T')$ with the same property such that  $\{G_t\mid t\in T\}\cup \{1\}=\{G_t\mid t\in T'\}\cup \{1\}$  is $\C$-projective (see \cite[Proposition 5.7]{HZ}).\end{rem}

\begin{rem}\label{projective C-projective} If $\C_2\subset \C_1$ are two classes of finite groups closed for subgroups, quotients and extensions then  a  $\C_1$-projective pro-$\C_2$ pile $(G,T)$ is $\C_2$-projective, since every finite $\C_2$-embedding problem is a $\C_1$-embedding problem. In particular, if $G$ is pro-$\C$ and $(G,T)$ is projective, then $(G,T)$ is $\C$-projective for any class of finite groups closed for subgroups, quotients and extensions. Similarly if $G$ is projective relative to $\G$ then it is $\C$-projective relative to $\G$.

\end{rem}

\begin{prop}\label{pro-p pile} Let  $G$ be a {\it $\C$-projective} relative to $\G$ profinite group and $T_0=\{t\in T\mid G_t\neq 1\}$. If $G$ is finitely generated and  $G_t, t\in T$ are all pro-$p$, then $T_0$ is closed in $T$ and $|T_0/G|\leq d$, the minimal number of generators of $G$. Moreover, $\sum_t d(G_t)\leq d(G)$, where $t$ runs through representatives of $G$-orbits in $T_0$.

\end{prop}

\begin{proof} Suppose on the contrary $\sum_t d(G_t)> d(G)$.  Then there exists an epimorphism $\varphi:G\longrightarrow A$ to a  finite group $A$ such that there  are  non-trivial non-conjugate subgroups $G_{t_1}, \ldots, G_{t_{n}}$ with the images $A_1=f(G_{t_1}), \ldots, A_{n}=f(G_{t_{n}})$, none of them is contained in a
conjugate of the other  and with $\varphi(\bigcup_{t\in T} G_t)=\bigcup_{i=1}^n A_i$ such that $\sum_i d(A_i)> d(G)$. Define $\A=\{\varphi(G_t)\mid {t\in T}\}$. Let $f:F_\C\longrightarrow A$ be an epimorphism from a free pro-$\C$ group   $F_\C$.  Form  a free pro-$\C$ product $B=\coprod_{i=1}^n{}^\C A_{i}\amalg{}^\C F_\C$. Put  $\B=\{A_{i}^b\mid b\in B\}$.  Let $\alpha:\B\longrightarrow  \A$ be a rigid morphism of pairs defined by sending $A_i$ to their copies in $A$.
Since $G$ is $\C$-projective reative to $\G$, the  embedding problem 
\begin{equation}\label{eqprosoluble}
\xymatrix{&&(G,\G)\ar[d]^{\varphi}\ar@{-->}[dll]_{\gamma}\\
             (B,\B)\ar[rr]_{\alpha}&& (A,\A)}
\end{equation} 
admits a solution $\gamma:(G,T)\longrightarrow (B,Y)$ (see \cite[Lemma 5.3.1]{HJ}).
Note that $\gamma(G_{t_i})$ is conjugate to $A_i$. So the natural epimorphism $B\longrightarrow \prod_{i=1}^n A_i$ restricts surjectively on $\gamma(G)$. As $A_i$ are non-trivial finite $p$-groups, the minimal number of generators of $\prod_{i=1}^n A_i$ is greater than $d$, a contradiction. Thus $\sum_t d(G_t)\leq d(G)$, $T_0/G$ is finite and so $T_0$ is closed in $T$. This finishes the proof.

\end{proof}

If $G_t$ are not pro-$p$ then $|T/G|$ can be greater than the minimal number of generators of $G$ (see \cite[Theorems 9.5.3 and 9.5.4]{RZ}), but we do not know whether it is always finite.

\bigskip 

We shall frequently use the following lemma on relative projectivity of subgroups of free pro-$\C$ products from \cite{HZ}. The statement is slightly different, but the proof is the same. 

\begin{lemma}\cite[Lemma 5.12]{HZ}\label{sub free prod}
Let $G=(\coprod_{x\in X}{}^\C G_x) \coprod{}^\C P$ be a free pro-$\C$ product of a continuous family of its subgroups $G_x$ and a projective pro-$\C$ group $P$. 
Let $H$ be a closed subgroup of $G$.
Put $$\calH = \{H \cap G_x^g \mid x \in X,\ g \in G\}.$$
Then there is a profinite $H$-space $T$ such that
\begin{itemize}
\item [(a)]
$\calH = \{H_t \mid t \in T\}$;

\item [(b)] $H_t\cap H_{t'}=1$ for $t\neq t'$;

\item [(c)] $\bfH = (H,T)$ is a $\C$-projective pile.

\item [(d)]
$G$ is projective relative to $\G$.
\end{itemize}
  Moreover,   $T$ is the disjoint union $T=\bigcup_{x\in X} G_x\backslash G$  on which $H$ acts by multiplication on the right viewed as a quotient space  of $X\times G$ via the map $X\times G\longrightarrow T$ given by $(x,g)\rightarrow (x, G_xg)$.
\end{lemma}

\begin{proof} This is the construction of $T$ in \cite[Lemma 5.12]{HZ}.  
\end{proof}

\begin{rem}\label{second countable}
 Suppose $H$  in Lemma \ref{sub free prod} is second countable. Put $T_0=\{t\in T\mid G_t\neq 1\}$. Then by \cite[Proposition 5.4.2]{R} there exists  an epimorphism $\mu: \overline T_0\longrightarrow \widetilde T$ of $H$-spaces such that 

(a) $\widetilde T$ is second countable;

(b) $\mu_{|T_0}$ is an injection;

(c) $H_t=H_{\mu(t)}$ for every $t\in \overline T_0$;

(d) $\mu$ is a homeomorphism on the $H$-orbits of $\overline T_0$.  

\medskip
It follows then that $(H,\widetilde T)$ is a pile and $\calH\cup \{1\} = \{H_t \mid t \in \widetilde T\}\cup \{1\}$. Then  $(H,\widetilde T)$ is $\C$-projective by Remark \ref{matching definitions}.

\end{rem}

In the next proposition and in the rest of the section we shall use the external definition of a free pro-$\C$ product over a profinite space  as in  \cite[Section 5.1]{R} which is equivalent to Definition \ref{def:freeprod} in the sense that given a continuous system of pro-$\C$ groups $\{G_x\mid x\in X\}$ over a profinite space $X$ there exists a pro-$\C$ group $G$ containing $G_x$ as subgroups such that $G=\coprod_{x\in X}{}^\C\  G_x$ is a free pro-$\C$ product in the sense of Definition \ref{def:freeprod}.  The external definition is quite involved, based on the notion of profinite sheaves that are not used in the paper. Hence we shall use the notation of a free pro-$\C$ product as before even if we define a group as an external free pro-$\C$ product.

\begin{prop}\label{into free prod}   Let $(H,T)$ be a $\C$-projective pile such that there exists a continuous section $\sigma:T/H\longrightarrow T$. Then $H$ embeds into a free pro-$\C$-product $G=\coprod_{s\in\Sigma }{}^\C\  G_s\coprod{}^\C P$, where $\Sigma=Im(\sigma)$, $G_s$ is an isomorphic copy of $H_s$ and  $P$ is a projective pro-$\C$ group admitting a homomorphism $f:P\longrightarrow H$ such that $H=\langle f(P), H_s\mid s\in \Sigma\rangle$. Moreover, $\{H_t\mid t\in T\}=\{G_s^g\cap H\mid g\in G, s\in \Sigma\}$.

\end{prop}

\begin{proof} Let $\alpha:G\longrightarrow H$ be an epimorphism that extends $f$ and  sends $G_s$ onto their isomorphic copies in $H$. Put $S=\bigcup_{s\in \Sigma} G_s\backslash G$. We view $S$ as the quotient space of the profinite space $\Sigma \times G$
via the map $\pi \colon \Sigma \times G \to S$
given by
$(s,g) \mapsto (s, G_s g)$ (cf. Lemma \ref{sub free prod}).
By \cite[Proposition 5.2.3]{R}
this is a profinite space.

 By \cite[Construction 4.3]{HZ} $(G,S)$ is a pile and so extending $\alpha$ to $S\rightarrow T$ by $\alpha(G_sh)=s\alpha(h)$ one gets a rigid epimorphism $\alpha:(G,S)\longrightarrow (H,T)$. Since $(H,T)$ is $\C$-projective, by \cite[Proposition 5.4]{HZ}  the embedding problem \begin{equation}
\xymatrix{&& (H,T)\ar[d]^{id}\ar@{-->}[dll]_{\gamma}\\
             (G,S)\ar[rr]_{\alpha}&& (H,T)}
\end{equation}  has a solution   which is clearly an embedding.

\end{proof} 

\begin{coro}\label{embedding in prosoluble free product}  Let $G = \coprod_{x\in X} G_x$ be a free profinite product of  pro-$\C$ groups over a profinite space $X$. Then
every second countable pro-$\C$ subgroup $H$ of $G$ is isomorphic to a closed subgroup of the free pro-$\C$ product $$\coprod_{\sigma\in \Sigma}{}^\C\  H_\sigma \amalg{}^\C F_\C ,$$ where  each non-trivial  $ H_\sigma$ equals $H\cap G_x^{g}$  for some $g\in G$  and $F_\C$ is a second countable free pro-$\C$ group. \end{coro}

\begin{proof} Let $H$ be a second countable pro-$\C$ subgroup of $G$. By Remark \ref{second countable} there exists  a second countable profinite space $T$ on which $H$ acts continuously such that $(H,T)$ is a projective pile and $$\{H_t\mid t\in T\}\cup \{1\}= \{H \cap G_x^g \mid x \in X,\ g \in G\}\cup \{1\}.$$   So there exists a continuous section $\sigma: T/H\longrightarrow T$ (cf. \cite[Lemma 5.6.7]{RZ}). By Remark \ref{projective C-projective} $(H,T)$ is  $\C$-projective. Then from   Proposition \ref{into free prod} we get an embedding of $H$  into a free pro-$\C$ product $$\coprod_{\sigma\in \Sigma}{}^\C   H_\sigma\amalg^\C F_\C,$$
where $\Sigma=Im(\sigma)$ and $F_\C$ is  second countable free pro-$\C$.  

\end{proof}

\begin{theo}\label{prosoluble pile} Let $G$ be a profinite group, projective relative to a continuous family of subgroups $\G=\{G_t\mid t\in T\}$ closed for conjugation.  Then one of the following holds:

\begin{enumerate}

\item[(i)] There a prime $p$ such that all $G_t, t\in T$ are pro-$p$ and $G$ is $\calS$-projective relative to $\G$ (where $\calS$ is the class of all finite soluble groups).

\item[(ii)]   $G= \widehat\Z_\pi\rtimes C$ is a  profinite Frobenius groups (i.e., $C$ is cyclic, $|C|$ divides $p-1$ for any $p\in \pi$ and $[c,z]\neq 1$ for all $1\neq c\in C, 0\neq z\in \widehat \Z_\pi$). Moreover, $T_0=\{ t\in T\mid G_t\neq 1\}$ is closed in $T$  and $|T_0/G|=1$. Besides $\{ G_t\mid t\in T_0\}=\{C^z\mid z\in \widehat\Z_\pi\}$.

\item[(iii)] $|T_0|=1$ and $G_t=G$ for the unique $t\in T_0$.

\end{enumerate}  

\end{theo}

\begin{proof}

 Suppose (iii) does not hold. If $G$ is a profinite Frobenius group, then $G=K\rtimes C$, where $C$ is finite (see \cite[Theorem 4.6.9]{RZ}) and $K$ and $C$ are coprime. By \cite[Lemma 4.6.5]{RZ} $G=\varprojlim G/V $, where $G/V=(K/V)\rtimes (HV/V)$ and $V$ runs through the collection of all open normal subgroups of $G$ contained in $K$. Put $G_V=G/V$ and  let $\varphi_V:G\longrightarrow G/V$ be the natural epimorphism. Let $A_1, \ldots A_n$ be the set of representatives of the conjugacy classes of the set $\A=\{\varphi(G_t), t\in T\}$ of images of $G_t$ in $G_V$.  Let $f:F\longrightarrow G_V$ be an epimorphism of a free profinite group of finite rank to $G_V$. Form a free profinite product $B=\coprod_{i=1}^n A_{i}\amalg F$ and let $\B=\{A_i^b\mid b\in B\}$.  Thus one has a rigid  embedding problem

\begin{equation}\label{frobenius}
\xymatrix{&&(G,\G)\ar[d]^{\varphi}\ar@{-->}[dll]_{\gamma}\\
             (B,\B)\ar[rr]_{\alpha}&& (G_V,\A)}
\end{equation} 
where $\alpha$ extends  $f$ and the natural isomorphisms of $A_i$ to their copies in $G/V$. 
Since $G$ is   projective relative to $\G$, and hence by Remark \ref{second countable embedding} so is \eqref{frobenius} as $B$ is second countable. This means that  there exists a solution $\gamma$ of \eqref{frobenius}. Then $\gamma(G)$ is a profinite Frobenius group (as $ker(\gamma)\leq K$ and so one can apply \cite[Theorem 4.6.8]{RZ} to $\gamma(G)$). By \cite[Corollary 7.1.8 (c2)]{R} (a profinite Frobenius group can only appear in (a) and (c2) of this corollary, but (a) cannot hold by assumption on $G_V$) the group  $\gamma(G)$ is a semidirect product of a cyclic projective group and cyclic finite group, i.e. $\gamma(G)=\gamma(K)\rtimes \gamma(C)$ with $\gamma(K)$ projective cyclic and $\gamma(C)$ finite cyclic. As it happens for every $V$ we deduce that $K\cong \widehat\Z_\pi$ and $C$ is finite cyclic. As $G_t\cap G_{t'}=1$ for $t\neq t'$ (see \cite[Proposition 5.7]{HZ}) the normalizer of any subgroup $H_t\leq G_t$ is contained in $G_t$. Hence $K\cap G_t=1$ for every $t\in T$ and since we assumed that $G\neq G_t$, we deduce that $G_t$ are finite cyclic.  
 
By the profinite version of Schur-Zassenhaus theorem (\cite[Theorem 2.3.15]{RZ}) all complements of  $\widehat \Z_\pi$ in $G$ are conjugate so that $G_t\cong C$ for each $t\in T_0$ and  $|T_0/G|=1$. It follows that   $\{ G_t\mid t\in T_0\}=\{C^z\mid z\in \widehat\Z_\pi\}$. Thus (ii) holds in this case.
 
 \medskip
 Suppose now that $G$ is not Frobenius. We show that for some prime $p$,  $G_t$ is pro-$p$ for all $t\in T$. Suppose on the contrary $G_t$ are  not all pro-$p$ for any prime $p$.   Then there exists an epimorphism $\varphi:G\longrightarrow A$ to a finite group $A$ such that $A$ is non-cyclic  and not Frobenius (see \cite[Theorem 4.6.9 (a),(c)]{RZ}), and the images of $\varphi(G_{t}), t\in T$ in $A$ are not all finite $p$-groups for any prime $p$. As before let $A_1, \ldots A_n$ be the set of representatives of the conjugacy classes of the set $\A=\{\varphi(G_t), t\in T\}$ of images of $G_t$ in $A$.   Let $f:F\longrightarrow A$ be an epimorphism of a free profinite group of finite rank to $A$.  Form a free profinite product $B=\coprod_{i=1}^n A_{i}\amalg F$ and let $\B=\{A_i^b\mid b\in B\}$.  Thus one has a rigid  embedding problem

\begin{equation}\label{free}
\xymatrix{&&(G,\G)\ar[d]^{\varphi}\ar@{-->}[dll]_{\gamma}\\
             (B,\B)\ar[rr]_{\alpha}&& (A,\A)}
\end{equation} 
where $\alpha$ extends $f$ and the natural isomorphisms of $A_i$ to their copies in $A$. 
Since $G$ is   projective relative to $\G$, by Remark \ref{second countable embedding} so is \eqref{frobenius} as $B$ is second countable. This means that there exists a solution $\gamma$ of \eqref{free}. But $\gamma(G)$ is Frobenius by Theorem \ref{prosoluble}, contradicting that $\alpha(\gamma(G))=\varphi(G)$ is not. 

Thus all $G_t$s are pro-$p$.  It remains to show that $G$ is $\calS$-projective relative to $\G$. But this follows from Remark \ref{projective C-projective}.
The proof is finished. \end{proof}

\begin{lemma}\label{relative projectivity}
Let $G_s=\coprod_{i=1}^n{}^s K_i \amalg^s P$ be a free prosoluble product of finite $p$-groups $K_i$ and a  finitely generated  prosoluble projective group $P$. Then $G_s$ is projective relative to $\{K_i^g\mid g\in G_s, i=1, \ldots, n\}$.

\end{lemma}

\begin{proof} We observe first that $G_s$ is a subgroup of $L_s=\coprod_{i=1}^n{}^s (K_i\times C_i)$, where $C_i$ is a cyclic group of order $p^2$. Indeed, the kernel $U$ of the homomorphism $L_s\longrightarrow C_{p^2}$ that send all $C_i$  isomorphically to $C_{p^2}$ and all $K_i$ to 1 has the Kurosh decomposition $U=\coprod_{i=1}^n K_i \amalg^s F_s$, where $F_s$ is a free prosoluble group of rank $1+(n-1)p^2-n=(p^2-1)n-p^2+1\geq 2$ unless $n=1$ (see \cite[Theorem 9.1.9]{RZ}). So, if $n>1$,   $P$ embeds in $F_s$ (see \cite[Lemma 7.6.3]{RZ}) and we have an embedding of $G_s$ into $L_s$.

If $n=1$, then  $K_1\amalg^s P$ embeds in $K_1\amalg^s K_1\amalg^s P$ so that we can use the previous embedding.

Hence by Lemma \ref{sub free prod} we may assume w.l.o.g. that $G_s=\coprod_{i=1}^n{}^s K_i$. By \cite[Theorem 1.4]{EH} $G_s$ retracts to a free profinite product  $\coprod_{i=1}^n K_i$ such that the images of $K_i$ are conjugate to the corresponding free factors and so applying Lemma \ref{sub free prod} again together with the fact that any maximal finite subgroup of $G_s$ is conjugate to one of $K_i$s (cf. \cite[Corollary 7.1.3]{R}), we deduce the result.
\end{proof}

{ \it Proof of Theorem \ref{relatively projective}}.

Suppose $G$ is projective relative to $\G=\{G_t, t\in T\}$. Then, by Theorem \ref{prosoluble pile},  items (i) or (ii) or (iii) hold.

\medskip
Converse. Suppose (i) or (ii) or (iii) hold. If (iii) holds, then $G$ is clearly projective relative to $\{G_t\mid t\in T\}$. If (ii) holds, then by \cite[Corollary 5.2]{GH} $G$ embeds into a free profinite product $G_0=C\amalg C$.  By Lemma \ref{sub free prod} $G$ is projective relative to $\G$. 

Suppose now (i) holds. To prove that $G$ is projective relative to $\G$, we need to solve a finite rigid  embedding problem 
\begin{equation}\label{finite}
\xymatrix{&&(G,\G)\ar[d]^{\varphi}\ar@{-->}[dll]_{\gamma}\\
             (B,\B)\ar[rr]_{\alpha}&& (A,\A)}
\end{equation} 

Choose representatives $A_1,A_2,\ldots A_n$  of the conjugacy classes of subgroups in  $\A$ and  subgroups $B_i\in \B$  such that $\alpha(B_i)=A_i$. Let $f:F\longrightarrow B$ be an epimorphism of a free profinite group $F$ of finite rank to $B$. As $A$ is soluble, $\alpha f$ factors through    a free prosoluble group $F_s$ of the same finite rank as $F$, i.e. there are epimorphisms $\eta:F\longrightarrow F_s$ and $f_s:F_s\longrightarrow A$ such that $\alpha f=f_s\eta$.  Form a free profinite product $B'=\coprod_{i=1}^n B_{i}\amalg F$ and a free prosoluble product $B_s=\coprod_{i=1}^n{}^s A_{i}\amalg^s F_s$. Put $\B'=\{ B_i^b\mid b\in B'\}$ and $\B_s=\{ B_i^b\mid b\in B_s\}$.  Then one has the following commutative diagram:

\begin{equation}\label{finite free}
\xymatrix{&&(G,\G)\ar[dd]^{\varphi}\ar@{-->}[dl]_{\gamma_s}\\
(B',\B')\ar[d]_{\pi}\ar[r]^{\rho}& (B_s, \B_s)\ar[dr]^{\alpha_s}&\\
             (B,\B)\ar[rr]_{\alpha}&& (A,\A)}
\end{equation} 
where $\pi:B'\longrightarrow B$ and $\alpha_s:B_s\longrightarrow B$ are the epimorphisms that send $B_i$ to their copies in $B$ and $\pi_{|F}=\eta$.
Since $G$ is $\calS$-projective relative to $\G$, the  embedding problem $(\varphi, \alpha_s)$
admits a solution $$\gamma_s:(G,\G)\longrightarrow (B_s,\B_s)$$ (see Remark \ref{second countable embedding}) so that $\alpha_s\gamma_s=\varphi$. By Lemma \ref{relative projectivity} $B_s$ is projective relative to $\B_s$ and so using \cite[Lemma 5.3.1]{HJ} again  the following embedding problem

\begin{equation}\label{free products}
\xymatrix{&&(B_s,\B_s)\ar[d]^{id}\ar@{-->}[dll]_{r}\\
             (B',\B')\ar[rr]_{\rho}&& (B_s,\B_s)}
\end{equation} 
has a solution $r$, so that $\rho r=id$.    
Then $\pi r\gamma_s$ is a solution of  \eqref{finite}, because $\alpha\pi r\gamma_s=\alpha_s\rho r\gamma_s=\alpha_s\gamma_s=\varphi$.
This finishes the proof.

\bigskip
The converse implication from (i) is in fact a generalization of Lemma \ref{relative projectivity} that we shall state as a

\begin{coro}\label{generalization relative projectivity} Let $(G,\G)$ be a  profinite pair such that $G_t$ are pro-$p$ and $G$ is prosoluble. If $G$ is $\calS$-projective relative to $\G$ then $G$ is projective relative to $\G$.

\end{coro}

It shows that for prosoluble groups the notions of  projectivity and prosoluble projectivity relative to family of pro-$p$ subgroups coincide. 

\bigskip 
{\it Proof of Corollary \ref{finitely generated}} 
 follows from Proposition \ref{pro-p pile} and Theorem \ref{relatively projective}.

\bigskip

{\it Proof of Corollary \ref{herfort}}. 
By  Corollary \ref{generalization relative projectivity}   $G_s$ is projective relative to $\G_s=\{ G_x^g\mid g\in G_s\}$. Note that $\alpha$ extends to a rigid epimorphism of pairs $\alpha:(G,\G)\longrightarrow (G_s,\G_s)$, where $\G=\{ G_x^g\mid g\in G\}$.    Hence   by  Remark \ref{second countable embedding} the embedding problem 

\begin{equation}\label{ soluble an normal free products}
\xymatrix{&&(G_s,\G_s)\ar[d]^{id}\ar@{-->}[dll]_{r}\\
             (G,\G)\ar[rr]_{\alpha}&& (G_s,\G_s)}
\end{equation} 
has a solution which is the needed retraction.

\begin{theo}\label{characterization} Let $G=\coprod_{x \in X} G_x$ be  a  free profinite product  over a profinite space $X$. A second countable prosoluble group $H$  is isomorphic to a subgroup of $G$  if and only if one of the following holds:
 
 \begin{enumerate} 
 
 \item[(i)]  $H$ is isomorphic to a subgroup of a second countable free prosoluble product  $H_s=\coprod_{t\in T}{}^s H_t$ of a continuous subsystem of pro-$p$ groups $\{H_t\mid t\in T\}$   of the continuous system $\{G_x^g\mid x\in X, g\in G\}$ over $Y=\bigcup_{x\in X} G_x\backslash G$, where $T$ is a second countable profinite space;
 \item[(ii)] $H$ is isomorphic to a Frobenius group  $\widehat \Z_\pi\rtimes C$, with $C$ a finite cyclic subgroup of $G_x$ for some $x\in X$;
\item[(iii)] $H$ isomorphic to a subgroup of  a free factor $G_x$.
\end{enumerate} 
 
 \end{theo}

\begin{proof}
Suppose $H$ is a subgroup of $G$. By Remark \ref{second countable} there  is a second countable profinite space $\widetilde T$ on which $H$ acts continuously such that $(H,\widetilde T)$ is a projective pile and $$\{H_t\mid t\in \widetilde T\}= \{H \cap G_x^g \mid x \in X,\ g \in G\}.$$ Then by Theorem \ref{relatively projective} either (ii) or (iii) hold, or all   $H_t$ are pro-$p$ and $H$ is $\calS$-projective relative to $$\{H \cap G_x^g \mid x \in X,\ g \in G\}.$$ 
As $\widetilde T$ is second countable,   $\widetilde T\longrightarrow \widetilde T/H$ admits a  continuous section $\sigma:\widetilde T/H\longrightarrow \widetilde T$  (see \cite[Lemma 5.6.7]{RZ}) and so by Proposition \ref{into free prod}  $H$ embeds into a free prosoluble product $\coprod_{t\in T}{}^s H_{t}\amalg^s F_s$ and a free prosoluble group $F_s$, where $T=im(\sigma)$. Now by Lemma \ref{embedding} $\coprod_{t\in T}{}^s H_t\amalg^s F_s$ embeds in $\coprod_{t\in T}{}^s H_t$ if $|T|>3$. Otherwise it embeds into $\coprod_{t\in T}{}^s H_t\amalg^s \coprod_{t\in T}{}^s H_t$ by Lemma \ref{embedding} again, that finishes the proof of this implication.

\bigskip

\bigskip
Converse. If (iii) hold there is nothing to prove.  
Suppose (ii) holds. Then  by \cite[Corollary 5.2]{GH} $H$ embeds into a free profinite product $C\amalg C$. 
As $G$ is not dihedral, By  \cite[Corollary 5.3.3]{R}  the normal closure is a free profinite product of copies of $G_x$ and so   $C\amalg C$ in turn embeds in $G$ as needed.

\medskip
 Suppose now (i) holds. 
 As any second countable profinite space embeds into second countable homogeneous, we may assume that $T$ is homogeneous and it suffices to embed $H_s$ into $G$. By Corollary \ref{herfort} $H_s$ retracts into $\coprod_{t\in T} H_t$ so it suffices to embed $\coprod_{t\in T} H_t$ into $G$.   Consider $G\times T$ as a continuous family $\{G\times \{t\}, \mid t\in T\}$ over $T$ and form a free profinite product $\coprod_{t\in T}(G\times\{t\})$. Then by \cite[Theorem 5.5.6]{R} 
 $\coprod_{t\in T} H_t$ embeds into $\coprod_{t\in T}G\times\{t\}$ and so it suffices to embed $\coprod_{t\in T}(G\times\{t\})$ into $G$. By Lemma \ref{getting free free factor} $G$ contains $G\amalg F$, where $F$ is a non-trivial free profinite group and so is homeomorphic to $T$.   
  By  \cite[Corollary 5.3.3]{R}  the normal closure  $N$ of $G$ in $G\amalg F$ is $\coprod_{f\in F} \coprod_{f\in F}  G^{f}\cong \coprod_{t\in T}(G\times\{t\}) $ as required.

   \end{proof}

\bigskip

{\it Proof of Theorem \ref{characterization fg}.}
 Only if is just a particular case of Theorem \ref{characterization} as well as the converse for the cases (ii) and (iii). 
 
 Moreover, suppose (i) of Theorem \ref{characterization fg} holds.  By Remark \ref{second countable} 
 there is a second countable profinite $H$-space $T$ such that
$$\{H \cap H_i^h \mid \ h \in H_s, i=1,\ldots, n\} = \{H_t \mid t \in T\}$$ and
 $\bfH = (H,T)$ is a $\calS$-projective pile. Then by Theorem \ref{relatively projective} $(H,T)$ is projective.  By Corollary \ref{finitely generated}  $T_0=\{t\in T\mid G_t\neq 1\}$ is closed and $T_0/H$ is finite and $\sum_t d(H_t)\leq d(H)$. 
 
 It follows that $\{H_t\mid t\in T_0\}$ is a continuous subfamily of $\{G_x^g\mid x\in X, g\in G\}$  so we can apply Theorem \ref{characterization} again to finish the proof.

\section{Questions \ref{subgroup} and \ref{retract}}

\bigskip

\begin{prop}\label{single projective} Let $G_s=A \amalg^s P$ be a prosoluble free product of a pro-$p$ group $A$ and a prosoluble projective group $P$. Then $(G_s, T)$ is projective, where $T=A\backslash G_s$.
\end{prop}

\begin{proof}   By Lemma \ref{sub free prod} $(G_s , T)$ is an $\calS$-projective pile, where $T=A\backslash G_s$. So by Corollary \ref{generalization relative projectivity}  $(G_s, T)$ is projective.    
\end{proof}

The next Corollary answers Question \ref{retract}.

\begin{coro}\label{retracts} Let $A$ be a pro-$p$ group and $P$ a  prosoluble projective
	group. Then $G=A \amalg P$ admits a prosoluble retract $G_s=A \amalg^s P$.
\end{coro} 

\begin{proof} By Proposition \ref{single projective} $(G_s, T)$ is a projective pile, where $T=A\backslash G_s$. But $|T/G_s|=1$ and so by Proposition \ref{into free prod} $G_s$ embeds into $G=A\amalg P$.

\end{proof}

\bigskip

We finish the paper  answering  Question \ref{subgroup} in the second countable case, but assuming that the free factors are prosoluble rather than pro-$p$. 

\begin{theo}\label{subgroups}  Let $G = G_1\amalg G_2$ be a free profinite product of  prosoluble groups $G_1$ and $G_2$. Then
every second countable prosoluble subgroup of $G$ is isomorphic to a closed subgroup of $L=G_1 \amalg^s G_2$.\end{theo}

\begin{proof} If $G_1, G_2$ are of order $\leq 2$ then $G=L$ and there is nothing to prove. Suppose both $G_1$ and $G_2$ are non-trivial and $|G_1|+|G_2|>4$.

 Let $H$ be a prosoluble subgroup of $G$.   Then by Corollary \ref{embedding in prosoluble free product} $H$ embeds into a free prosoluble product $\coprod_{i=1}^2{}^s\coprod_{g\in D_i}{}^s  (H\cap G_i^g)$, where $g$ runs through the closed set $D_i$ of representatives of the double cosets $G_i\backslash G/H$, $i=1,2$. This embeds in turn into  $\coprod_{i=1}^2{}^s\coprod_{g\in D_i}{}^s  G_i^g$.  Thus it suffices to embed $\coprod_{i=1}^2{}^s\coprod_{g\in D_i}{}^s  G_i^g$ into $L$.

 Note that $|G|=|L|$. By Lemma \ref{embedding} $L$ contains  a free prosoluble product $G_1 \amalg^s G_2^l \amalg^s U$ for some $l\in L$, where $|U|=|L|$. Thus $|G|=|L|=|U|=2|U|$.  Therefore  there exists a homeomorphism $\sigma:G\longrightarrow U_1\cup U_2$, where $U_1, U_2$ are two copies of $U$.  By  \cite[Corollary 5.3.3]{R}  the normal closure of $G_1 \amalg^s G_2^l$ in $G_1 \amalg^s G_2^l \amalg^s U$ is $$\coprod_{u\in U}{}^s (G_1 \amalg^s G_2^l)^u=\coprod_{i=1}^2{}^s\coprod_{u\in U_i}{}^s  G_i^{l_iu}, $$ where $l_1=1,l_2=l$. Thus  there is an isomorphism $\coprod_{i=1}^2{}^s\coprod_{g\in G}{}^s  G_i^g\longrightarrow \coprod_{i=1}^2{}^s\coprod_{u\in U_i}{}^s G_i^{l_iu}$ induced by a homeomorphism $G\longrightarrow U_1 \cup U_2$ and  isomorphisms $id:G_1\longrightarrow G_1$,  $G_2\longrightarrow G_2^l$. This finishes the proof.   

\end{proof}

\noindent{\bf Acknowledgment.}

The author is grateful to Pavel Shumyatsky who gave the argument for Lemma \ref{frattini argument1}. The author also thanks  Dan Haran and Wolfgang Herfort   for comments to improve considerably the text of the paper.

\bigskip

\noindent
Financial Support: Partially suppoted by CNPq and FAPDF.

\bigskip

\noindent \href{https://www.mat.unb.br/pessoa/152/Pavel-Zalesski}{Pavel A. Zalesskii}, Universidade de Brasília, Departamento de Matemática, 70910-900 Brasília DF, Brazil. {\it Email address}: \href{mailto:pz@mat.unb.br}{pz@mat.unb.br}.

\end{document}